\documentclass[12pt]{amsart} 

\usepackage{amssymb}
\usepackage[pdftex]{graphicx}

\newtheorem{thmA}{Theorem}

\newtheorem{theorem}{Theorem}[section]
\newtheorem{thm}[theorem]{Theorem}
\newtheorem{lemma}[theorem]{Lemma}
\newtheorem{cor}[theorem]{Corollary}
\newtheorem{prop}[theorem]{Proposition}
\newtheorem{addendum}[theorem]{Addendum}

\theoremstyle{definition}
\newtheorem{definition}[theorem]{Definition}

\theoremstyle{remark}
\newtheorem{remark}[theorem]{Remark}

\numberwithin{equation}{section}

\def\E{\Lambda}

\def\P{\Pi}
\def\<{\langle}
\def\>{\rangle}
\def\A{\mathcal A}
\def\R{\mathcal R}
\def\Y{\mathcal Y}
\def\H{{\rm{H}}}

\def\th{\theta}
\def\Th{\Theta}

\def\C{\mathcal C}

\def\La{\Lambda}

\def\cech{\check}

\def\T{T}
\def\U{U}

\def\a{\alpha}
\def\b{\beta}
\def\c{\gamma}

\def\-{\underline}
\def\Z{\mathbb Z}
\def\N{\mathbb N}
\def\G{\Gamma}
\def\S{\Sigma}

\def\ln{\<\!\<}
\def\rn{\>\!\>}

\catcode`\@=11
\def\serieslogo@{\relax}
\def\@setcopyright{\relax}
\catcode`\@=12

\begin{document}

\title{Decision problems and profinite completions of groups}
 
\author[Martin R. Bridson]{Martin R.~Bridson}
\address{Mathematical Institute, 24-29 St Giles', Oxford OX1 3LB, UK} 
\email{bridson@maths.ox.ac.uk} 

\date{8 April 2008; 18 September 2008}



\thanks{This research was supported
by a Senior Fellowship from the EPSRC of Great Britain.}

\begin{abstract}
We consider pairs of finitely presented, residually finite groups $P\hookrightarrow
\G$ for which the induced map of profinite completions $\hat P\to \hat\G$ is an
isomorphism. We prove that there is no algorithm that, given an arbitrary such
pair, can determine whether or not $P$ is isomorphic to $\G$.
We construct pairs for which the conjugacy problem in $\G$
can be solved in quadratic time but the conjugacy problem in $P$ is unsolvable.

Let $\mathcal J$ be the class of super-perfect groups that 
have a compact classifying space and no proper subgroups of finite index. 
We prove that there does not exist an algorithm that, given a 
finite presentation of a group $\G$ and a guarantee  that $\G\in\mathcal J$,
can determine whether or not $\G\cong\{1\}$.

We construct a finitely presented acyclic group $\H$ and an integer $k$
such that there is no algorithm that can determine which $k$-generator
subgroups of $\H$ are perfect.
\end{abstract}

\maketitle

\centerline{\em{For Karl Gruenberg, in memoriam}}

\smallskip

\section*{Introduction}

The profinite completion of a group $\G$ is the inverse
limit of the directed system of
finite quotients of $\G$; it is denoted $\hat \G$. 
The natural
map $\G\to \hat \G$ is injective if and only if $\G$ is residually finite. In \cite{BG}  
Bridson and Grunewald settled a question of Grothendieck \cite{groth} by
constructing pairs of finitely presented, residually finite groups $u: P\hookrightarrow\Gamma$ such
that $\hat u :\hat P\to\hat\Gamma$ is an isomorphism but $P$ is not isomorphic to (or even
quasi-isometric to) $\Gamma$. Pairs of finitely generated groups with this property had
been found earlier by Platonov and Tavkin \cite{PT}, Bass and Lubotzky \cite{BL},
and Pyber \cite{pyber}.  

In the present article we begin to explore how different a pair of 
finitely presented residually finite groups $u:P\hookrightarrow\G$ can be if
$\hat u :\hat P\to\hat\Gamma$ an isomorphism, and ask how hard it can be to determine
if $u$ is an isomorphism. To make precise the distinctions involved, we focus
 on the basic decision problems of group theory.
 Since all finitely generated, residually finite groups has a solvable
word problem, the first serious issue to be addressed is the solvability of the
conjugacy problem. 

\begin{thmA}\label{t:conj}
There exist pairs of finitely presented, residually finite groups $u:P\hookrightarrow\G$
such that $\hat u:\hat P\to \hat \G$ is an isomorphism but the conjugacy problem in $P$
is unsolvable while the conjugacy problem in $\G$ can be solved in quadratic time.
\end{thmA}

Refinements of this result are discussed in Section \ref{s:conj}. In particular
the  ``quadratic time" assertion is sharpened and made more precise, and the issue
of whether $\G$ can be made conjugacy separable is discussed.

\smallskip  

Concerning the isomorphism problem we prove the following result, a  more technical
statement of which appears as Theorem \ref{t:iso1}.

\begin{thmA}\label{t:iso}
There does not exist an algorithm that, given a 
finitely presented, residually finite group $\G$ and a finitely
presentable subgroup $u:P\hookrightarrow\G$ with
$\hat u:\hat P\to \hat \G$ an isomorphism, can determine whether or
not $P$ is (abstractly) isomorphic to $\G$, nor does there exist an algorithm that 
can determine whether or not $u$ is an isomorphism.
\end{thmA}
 
As an illustration of the difficulty of deciding which subgroups of a fixed residually finite
group are profinitely equivalent to the ambient group we prove:

\begin{thmA}\label{t:inj}
There exists a residually finite, finitely presented group $\G$, an
integer $k$, a finitely presented subgroup  $u:P\hookrightarrow\G$ with 
$\hat u:\hat P\to \hat \G$ an isomorphism, and a recursive sequence of 
$k$-generator subgroups $\iota_n: K_n\hookrightarrow \G$ with
$P\subset K_n\subset \G$ such that there is no algorithm that can decide for which
$n$ the map $\hat \iota_n :\hat K_n\to\hat\G$ is an isomorphism.
\end{thmA}

Our proof of Theorem \ref{t:conj} is based on the constructions
in \cite{BBMS}, \cite{CM} and \cite{BG}, together with some observations
on the nature of universal central extensions (Propositions \ref{p:Gtilde} and \ref{p:WPtilde}). Theorems \ref{t:iso}
and \ref{t:inj} further
require the construction of new families of finitely presented groups,
well adapted for use in the framework of  \cite{BG} but wild enough to
encode undecidable phenomena (see Sections \ref{s:triv} and \ref{s:gen}). 
To this end, in Theorem \ref{t:acyclicFP} we  exhibit a finitely presented acyclic group $\H$ and an integer $k$
such there is no algorithm that can determine which $k$-generator
subgroups of $\H$ are perfect, while in Theorem \ref{t:triv1} we prove  a sharper
version of the following result. Recall that
a group $G$ is termed {\em{super-perfect}} if $H_1(G;\Z)=H_2(G;\Z)=0$.

\begin{thmA}\label{t:triv} Let $\mathcal J$ be the class of super-perfect groups that 
have a compact classifying space and no proper subgroups of finite index. 
There does not exist an algorithm that, given a 
finite presentation of a group $\G$ and a guarantee  that $\G\in\mathcal J$,
can determine whether or not $\G\cong\{1\}$.
\end{thmA} 

The proof of this theorem involves an algorithm that I learnt
from C.F.~ Miller III. Given 
a finite presentation of a perfect group, this algorithm
will output a finite presentation
of the universal central extension of the group;   see Corollary  \ref{c:algo}.

\smallskip

Karl Gruenberg asked me if Theorem \ref{t:iso} might be true 
after a London Algebra Colloquium that I gave in March 2004 on my work with
Fritz Grunewald \cite{BG}. Conversations
with Karl, who died in October 2007, were always stimulating, both mathematically and culturally, and he
is sorely missed.

\section{Background and Preparation}

In this Section we gather five sets of ideas that we will need in 
the proofs of the theorems stated in the introduction.  

\subsection{The Bridson-Grunewald Construction}

Most of the pairs  $u:P\hookrightarrow\G$
that we shall consider, with $\hat u$ an isomorphism,
 derive from the main construction in \cite{BG}. Thus we begin by recalling
 the salient points of that article.
 Recall that a finitely presented group $G$ is said to be of {\em{type $F_3$}}
if there is a CW-complex $K(G,1)$ that has fundamental group $G$,
contractible universal cover, and only finitely many cells in its
3-skeleton. 

\begin{theorem}[Bridson-Grunewald]\label{t:BG}
If $Q$ is a finitely presented group that is infinite but has
no non-trivial finite quotients, and if $H_2(Q,\Z)=0$, then there is a short exact sequence
of groups 
$$1\to N \to \G\overset{p}\to Q\to 1$$
where $\G$ is finitely presented, torsion-free,
hyperbolic and residually finite, 
$N$ is finitely generated
but not finitely presented, the inclusion $u:P\hookrightarrow\G\times\G$
of the  fibre product 
 $$P=\{(\gamma_1,\gamma_2)\mid p(\gamma_1)=p(\gamma_2)\}$$
induces an
isomorphism of profinite completions $\hat P\to \hat\G\times\hat\G$,
but $P$
is not isomorphic to $\G\times\G$.

If $Q$ is of type $F_3$, then $P$ is finitely presented.  
\end{theorem}

The assertion in the final sentence of this theorem is a special
case of the 1-2-3 Theorem of Baumslag, Bridson, Miller and
Short  \cite{BBMS}.  
The short exact sequence in the first sentence comes
from Wise's variation
on the Rips construction (see \cite{rips}, \cite{wise}), which is proved
using small cancellation theory. See Section 7 of \cite{BG} for a 
demonstration of the explicit nature of this construction.  

\begin{theorem}[Rips-Wise]\label{t:rips} There is an algorithm that
associates to every finite group-presentation $\mathcal Q\equiv \< X\mid R\>$ a finite
presentation $\mathcal G\equiv \<\cech X\mid \cech R\>$ with $\cech X = X\cup\{a_1,a_2,a_3\}$ and $|\cech R|= |R| + 6|X|$,
such that the group $\G$ with presentation $\mathcal G$ is residually finite, torsion-free, hyperbolic (in the sense of Gromov) and
has a compact 2-dimensional classifying space. The subgroup $N\subset\G$ generated by $\{a_1,a_2,a_3\}$
is normal and there is a short exact sequence
$$1\to N\to \G \overset{p}\to Q\to 1$$ defined by  $p(x)=x$ for all $x\in X$. 
\end{theorem}   

In order to make Theorem \ref{t:BG} useful, one needs a plentiful supply of groups
of type $F_3$ that are super-perfect and have no finite quotients. As explained
in Section 8 of \cite{BG}, one can obtain such a supply by proceeding
as follows. Following the constructions in \cite{mb-ep}, one can
embed any group $G$ of type $F_3$ in a group $G'$
 of type $F_3$ that has no non-trivial
finite quotients; one does this by forming suitable free products and HNN extensions. One 
then obtains a group $Q$ as needed in Theorem \ref{t:BG} by forming the 
universal central extension of $G'$. The constructions at the heart of the present article
implement this basic procedure in an algorithmic manner, using carefully chosen input groups $G$.

\subsection{Universal central extensions}\label{ss:H_2G}

The standard reference for universal central extensions is  \cite{milnor}, pp. 43-47.

A central extension of a group $G$ is a group $\tilde G$ equipped
with a  homomorphism $\pi:\tilde G\to G$ whose kernel is central
in $\tilde G$. Such an extension is {\em{universal}} if given
any other central extension $\pi':E\to G$ of $G$, there is a
unique homomorphism $f:\tilde G\to E$ such that $\pi'\circ f =\pi$.

The following proposition
summarizes those properties of universal central extensions 
that we need.

\begin{prop}\label{p:Gtilde}

\begin{enumerate}
\item
$G$ has a universal central extension $\tilde G\to G$  if and only if 
$G$ is perfect. (If it exists, $\tilde G\to G$ is unique up to isomorphism over $G$.)
\item
If $G$ is expressed as a quotient of a free group
$G=F/R$, then the natural map $[F,F]/[F,R]\to G$ is the universal
central extension of $G$. 
\item
The kernel of the map $[F,F]/[F,R]\to G$ is $R\cap [F,F]/[F,R]$,
which according to Hopf's formula is $H_2(G;\Z)$. Thus the universal
central extension of a super-perfect group $G$ is ${\rm{id}}:G\to G$. 
\item
If $\tilde G\to G$ is a universal central extension, then $\tilde G$
is super-perfect.
\item If $G$ has no non-trivial finite quotients then neither does $\tilde G$.
\item If $G$ is finitely presented then so is $\tilde G$.
\item If $G$ is of type $F_3$ then so is $\tilde G$.
\item If $G$ has a 2-dimensional classifying space $K(G,1)$ then
$\tilde G$ is torsion-free and has a compact classifying space.
\end{enumerate}
\end{prop}
 
 \begin{proof}The first four facts are standard; see \cite{milnor}, pp.~43-47.
Point (5) is proved by noting that  since
$\tilde G$ is perfect, a non-trivial finite quotient would have to be non-abelian, and factoring out
the centre would give a non-trivial quotient of $G$.

For items (6) and (7) we consider the short exact
sequence furnished by (2) and (3):
$$
1\to H_2(G;\Z) \to \tilde G \to G\to 1.
$$
If $G$ is finitely presented then $H_2(G;\Z)$
is finitely generated.
Finitely generated abelian groups have classifying spaces with
finitely many cells in each dimension, so in particular they are
of type $F_3$. For any $n$, if 
the first and last groups in a short exact sequence are of type
$F_n$ then so is the group in the middle (and likewise,
if they have compact classifying spaces then so does the group
in the middle) --- see \cite{ross}, Section 7.1.

Turning to (8), recall  that if a group has
a finite-dimensional classifying space, it is torsion-free (Remark \ref{tf}). Thus,
by consideration of the short exact sequence above, it suffices
to argue that $H_2(G;\Z)$ is torsion-free (in which case it has a compact
classifying space, namely a torus of the appropriate dimension).
Now, $H_2(G;\Z)$ is isomorphic
to the second homology group of the given $K=K(G,1)$, which is
assumed to have no 3-cells. Thus $H_2(G;\Z)$ is isomorphic to
the kernel of the second boundary map in the cellular chain complex of $K$;
in particular it is a free abelian group.
\end{proof}

\begin{remark}\label{tf} If a group $G$ has a finite dimensional $K(G,1)$
then $G$ must be torsion-free for otherwise it would contain an
element of prime order, $g$ say, and the quotient of $\tilde K$ by 
$\langle g\rangle \cong\Z_p$ would be a finite-dimensional $K(\Z_p,1)$,
contradicting the fact that $\Z_p$ has cohomology in infinitely
many dimensions.
\end{remark}

We shall also need to control the complexity of the word problem when
we pass to universal central extensions. This relies on the simple observation:

\begin{lemma}\label{l:Zfg} If $G$ is a finitely generated, recursively presented group 
and $H\subset G$ is a finitely generated subgroup with a solvable word problem, then
there is an algorithm that, given a word $w$
in the generators of $G$ and a guarantee that $w\in H$,  can determine whether or not $w=1$ in $G$.
\end{lemma}

\begin{proof} Fix finite sets of generators $A$ for $H$ and $X$ for $G$, and for each $a\in A$ fix a word
$u_a$ in $X^{\pm 1}$ such that $a=u_a$ in $G$. Given a word $w$ in the letters $X^{\pm 1}$ with
the promise that it defines an element of $H$, one runs through
products $\pi$ of the words $u_a$ and their inverses,
doing a na\"{\i}ve search on products $P$ of conjugates of the defining relations of $G$ to check
 if $P$ is equal to $w\pi$ in the free group $F(X)$. Running through all possibilities for $\pi$ and
$P$ along finite diagonals, one will eventually find a valid formula, proving that $w$ is equal in $G$ to
a certain word in the letters $A^{\pm 1}$. One can then use the solution to the word problem in
$H$ to decide whether or not $w=1$ in $G$.
\end{proof}

\begin{prop}\label{p:WPtilde}
If $G$ is a finitely presented perfect group whose centre is finitely generated,
then the word problem in $G$ is solvable if and only if the word problem in 
its universal central extension $\tilde G$
is solvable.
\end{prop}

\begin{proof} 
We fix  a finite generating set $X$ for $G$, choose a preimage $\tilde x\in \tilde G$
for each $x\in X$, and fix a finite generating set $Z$ for the centre of $\tilde G$  that
includes a finite generating set $Z'$ for the kernel of $\tilde G\to G$. Note that the
centre of $\tilde G$ is finitely generated because it is an extension of the centre of $G$
by $H_2(G;\Z)$. Let $\tilde X = \{\tilde x : x\in X\}$ and note that $\tilde X\cup Z'$ generates $\tilde G$.

First suppose that $\tilde G$ has a solvable word problem.
Given a word $w$ in the letters $X^{\pm 1}$, we replace each occurrence in $w$ of $x\in X$ by
$\tilde x$ and ask if the resulting word $\tilde w$ defines a central element of $\tilde G$.
The hypothesized solution to the word problem in $\tilde G$ allows us to decide this because
it is enough to check if $[\tilde w, y] =1$ for all $y\in \tilde X\cup Z'$. If $\tilde w$ is not central,
we stop and declare that $w\neq 1$ in $G$. If  $\tilde w$ is central in $\tilde G$ then
$w$ is central in $G$, and as in Lemma \ref{l:Zfg} this enables us to
decide whether or not $w=1$ in $G$.

Now suppose that $G$ has a solvable word problem. Given a word $W$
in the letters $(\tilde X\cup Z')^{\pm 1}$, we consider the word $w$ obtained
by deleting all occurrences of all letters $z\in Z'$. If $w\neq 1$ in $G$, then we
declare $W\neq 1$ in $\tilde G$. If $w=1$ then we know $W$ is central in $\tilde G$
and hence we can use Lemma \ref{l:Zfg} to determine whether or not $W=1$.
\end{proof}

\subsection{Aspherical presentations}\label{ss:ass}
We remind the reader that a presentation of a group $G$ is termed
{\em{aspherical}} if the standard 2-complex of the presentation is
a $K(G,1)$, that is, the
universal covering  is contractible. For example, a free presentation
of a free group is aspherical.  

It is straightforward to prove that 
if $\P_1\equiv\<\A_1\mid\R_1\>$ and $\P_2\equiv\<\A_2\mid\R_2\>$ are aspherical presentations of groups
$G_1$ and $G_2$, then the natural presentations
$\< \A_1\sqcup\A_2\mid \R_1\sqcup\R_2,\,u_iv^{-1}\, (i=1,\dots,n)\>$ of any amalgamated free product of the form $G_1\ast_F G_2$ with $F$
free of rank $n$, is also aspherical. Likewise, the natural presentations
of HNN extensions of the form $G_1\ast_F$ will be aspherical.

\subsection{Higman's group}\label{ss:higman} The following group
constructed by Graham Higman \cite{hig} will provide useful
input to certain of our constructions.
$$
J=\langle a,b,c,d \mid aba^{-1}=b^2,\, bcb^{-1}=c^2,\, 
cdc^{-1}=d^2,\, dad^{-1}=a^2\rangle.
$$ 
The salient features of $J$ are described in the following proposition,
in which $D$ is the
group with presentation
$D= \langle \a,\b,\c \mid \a\b\a^{-1}=\b^2,\, \b\c\b^{-1}=\c^2\>$.

\begin{prop}\label{p:hig}
\begin{enumerate}
\item $J$ has no non-trivial finite quotients.
\item The given presentation of $J$ is aspherical.
\item $J\cong D_1\ast_{F_2} D_1$, where $D_1\cong D_2\cong D$
and the amalgamation identifies the subgroups $\langle \a_i,\c_i\>
\subset D_i$  
 by $\a_1=\c_2$ and $\c_1=\a_2$.
  \item The abelianization of $\< a,b,c\>\subset J$ is infinite cyclic,
 generated by the image of $a$. 
\item $H_i(J;\Z)=0$ for all $i\ge 1$.
\end{enumerate}
 \end{prop}
 
 \begin{proof} 
 Higman \cite{hig} proved (1) using some elementary number theory.
 
 The first thing that we must prove for (3) is that the subgroup
 of $D$ generated by $\a$ and $\c$ is free of rank 2. But this is clear
 from Britton's Lemma, once one
 observes that $D$ is obtained from the infinite cyclic
 group generated by $\c$ by two HNN extensions along
 cyclic subgroups: the first extension has stable
 letter $\b$, the second has stable letter $\a$
and amalgamated subgroups in $\<\b\>$.
 Following the comments in Section \ref{ss:ass},
 this HNN description also shows that the given presentation of $D$
  is aspherical, and that the  presentation of
   $D_1\ast_{F_2} D_2$  displayed below  is too.
  
 The isomorphism in (3) is given by
 $a\mapsto \a_1,\ b\mapsto \b_1,\, c\mapsto \c_1,
 \, d\mapsto \b_2$ (which has an obvious inverse).
 This restricts to an isomorphism from $\< a,b,c\>\subset J$ 
 to $D_1$, which
has abelianisation $\Z$, generated by the image of $\a_1$.
 This proves (4). 
 
 In order to prove (2) we compare the given presentation of $J$
 to that associated with the description $J\cong D_1\ast_{F_2} D_2$, namely
 $$
 \<  \a_1,\b_1,\c_1, \a_2,\b_2,\c_2\mid \a_i\b_i\a_i^{-1}=\b_i^2,\, \b_i\c_i\b_i^{-1}=\c_i^2 \ (i=1,2),\, \a_1\c_2^{-1},\, \a_2\c_1^{-1}\>.
 $$
 The generators $\a_2$ and $\c_2$ together with the last two relators
 can be removed by obvious Tietze moves. These relators correspond
 to  bigons (2-cells with boundary cycles of length 2) in the standard
 2-complex of the presentation, and the effect of the Tietze moves is
to shrink each of these to an edge in the obvious manner. 
Since this shrinking is a homotopy equivalence, we conclude that the
modified presentation
is still aspherical. The given presentation of $J$ is obtained
from this modified presentation by simply renaming the generators, so (2) is proved.

Turning to (5), note that we already know that $H_1(J;\Z)$, the abelianization
of $J$, is trivial. Also,
in the light of (1), we know that the groups $H_i(J;\Z)$ coincide with the homology
groups of the cellular chain complex of the standard 2-complex of the given presentation of $J$.
The group of cellular $k$-chains $C_k$ is trivial if $k\ge 3$ (as there are no $k$-cells), 
so $H_k(J;\Z)=0$ for $k\ge 3$ and $H_2(J;\Z)$ is isomorphic to the kernel of
$\partial_2: C_2\to C_1$. The final point to note is that
 $\partial_2$  is injective because $H_1(J;\Z)=C_1/{\rm{im}}(\partial_2)=0$
and $C_1\cong C_2\cong\Z^4$.
 \end{proof}
 
The final argument in the above proof generalizes immediately
to prove the following lemma, which we shall need in Section \ref{s:gen}.

\begin{lemma}\label{l:H2} If $G$  has an aspherical presentation with $n$ generators
and $m$ relations, and the abelianization of $G$ is finite, then $H_2(G;\Z)\cong\Z^{m-n}$.
\end{lemma}

\subsection{Centralizers and conjugacy in hyperbolic groups}\label{h-conj}

Recall from \cite{gromov}
that a finitely generated group $\G$ is said to be
{\em{hyperbolic}} (in the sense of Gromov) if there is a
constant $\delta>0$ such that each side of each triangle in
the Cayley graph of $\G$ is contained in the $\delta$-neighbourhood
of the union of the other two sides. 

We shall only consider torsion-free hyperbolic groups.
The centralizer of every non-trivial element in such a group is cyclic;
in particular every element is contained in a maximal cyclic subgroup,
and the centralizer of a non-cyclic subgroup is trivial.

Hyperbolic groups admit a rapid and effective solution to the conjugacy problem,
both for individual elements and for finite subsets (but not finitely generated subgroups).
In particular, Bridson and Howie \cite{BH} prove that if  a torsion-free group $\G$ is $\delta$-hyperbolic with
respect to a generating set of cardinality $k$, then there is a constant $C=C(\delta,k)$
and an algorithm that, given two lists of words $[a_1,\dots,a_m]$ and
$[b_1,\dots,b_m]$ in the generators of $\G$, the longest word having length $\ell$,
will terminate after at most
$Cm\ell^2$ steps having determined whether or not there exists an element  $\gamma\in\G$
such that $\gamma a_i\gamma^{-1}=b_i$ for $i=1,\dots,m$, outputting such a $\gamma$
if it exists. Applying this to lists of length one, we deduce
that the conjugacy problem for individual elements  can
be solved in quadratic time in $\G$, and indeed in the direct product of finitely many copies of
$\G$.  
A sharper result was obtained by 
Epstein and Holt \cite{EH}. They prove
 that the conjugacy problem in $\G$ is solvable in linear time if
one uses a standard RAM model of computation in which basic arithmetical operations on integers are assumed to take place in constant time; this gives an algorithm of Turing complexity $O(n\, \log n)$.
This algorithm extends easily to $\G\times\G$.

To mark the distinction between Turing and RAM models of computation, we shall use the
term {\em{RAM-linear}} when referring to the Epstein-Holt algorithm.

\section{The conjugacy problem}\label{s:conj}

With the preceding discussion in hand, we can state a more
precise version of Theorem \ref{t:conj}.

\begin{thm}\label{t:conj1}
There exist pairs of finitely presented, residually finite groups $u:P\hookrightarrow G$
such that $\hat u:\hat P\to \hat G$ is an isomorphism but the conjugacy problem in $P$
is unsolvable while the conjugacy problem in $G$ can be solved in RAM-linear
time.
\end{thm}

\begin{proof} 
In \cite{CM} (cf.~section \ref{ss:CM} below)
Collins and Miller described a group $G$ that has an unsolvable
word problem and a compact 2-dimensional
classifying space $K(G,1)$; in particular $G$ is torsion-free.
As in \cite{mb-ep} (or Section \ref{ss:noFin} below),
by forming suitable free products
with amalgamation and HNN extensions along free subgroups,
we can embed $G$ in a group $G'$ that again has such a 
classifying space and has
the additional feature that $G'$ has no proper subgroups of finite
index. By construction, $G'$ has trivial centre. And since $G'$
contains a copy of $G$, it has an unsolvable word problem.

Let $Q$ be the universal central extension of $G'$. 
We saw in Proposition \ref{p:Gtilde} that $Q$ is 
super-perfect, torsion-free, has no proper subgroups of finite
index, and has a compact classifying space. In Proposition \ref{p:WPtilde}
we proved that the word problem in $Q$ is unsolvable.

Let $\G$ be the torsion-free hyperbolic group obtained by applying Theorem \ref{t:rips}
 to $Q$ and let $P\subset\G\times\G$ be the 
fibre product of $\G\to Q$. Theorem \ref{t:BG}
assures us that $P$ is finitely presented and that
$P\to\G\times\G$ induces an isomorphism of profinite completions.
We saw in Section \ref{h-conj} that the conjugacy problem 
in $\G\times\G$ can be solved in RAM-linear  time, so we will be
done if we can argue that the conjugacy problem in $P$ is
unsolvable.

In the notation of Theorem \ref{t:BG}, we have
$p:\G\to Q$ with kernel $N$. We fix
a finite generating set $X$ for $Q$, choose a set of lifts
$\tilde X = \{\tilde x \mid  x\in X\}$ with $p(\tilde x) = x$ and consider
a generating set $\hat X\cup \{a_1,\dots,a_l\}$ for $\G$, where each
$a_i$ is in $N$ and the cyclic subgroups $C_i:=\<a_i\>\subset N$ are maximal
with respect to inclusion. Note that since $Q$ is torsion
free, $\<a_i\>$ will be maximal in $\G$ not just in $N$, and
hence the centralizer $Z(a_i)$ of $(a_i,a_i)$ in $\G\times\G$ is
$C_i\times C_i$. In particular $Z(a_i)\subset P$.

We claim that there is no algorithm that, given a word $w$ in 
the letters $\tilde X^{\pm 1}$, can determine whether or not
$(w^{-1}a_1w,a_1)$ is conjugate to $(a_1,a_1)$ in $P$.

First note that $(w^{-1}a_1w,a_1)$ does indeed belong to $P$
since $p(w^{-1}a_1w) = p(a_1)=1$. Also observe that
$(w,1)$ conjugates $(a_1,a_1)$ to $(w^{-1}a_1w,a_1)$ in $\G\times\G$.
Hence there exists $\pi\in P$ with $\pi (a_1,a_1)\pi^{-1}=
(w^{-1}a_1w,a_1)$ if and only if $\pi (w,1)\in Z(a_1)\subset P$,
which is equivalent to $(w,1)\in P$. 
But $(w,1)$ belongs to $P$ if and only if $w$ (viewed
now as a word in the letters $X$) is equal to $1$ in $Q$. And since
$Q$ has an unsolvable word problem, there is no algorithm that
can determine whether or not $(w,1)\in P$.
\end{proof}

\begin{remark} In the course of the above
proof we established that the membership problem for $P$
is unsolvable; cf.~\cite{BBMS}, p.468.
\end{remark} 

Each of the non-trivial groups $Q_n$ constructed in the proof of Theorem \ref{t:triv1} can play the
role of $Q$ in the above proof. Thus we have:

\begin{addendum} There does not exist an algorithm that, given a finitely presented,
residually finite group $G$ with a conjugacy word problem and a finitely presentable
subgroup $u:P\hookrightarrow G$ with $\hat u: \hat P\to \hat G$ an isomorphism, can
determine whether or not $P$ has a solvable conjugacy problem.
\end{addendum}

\subsection{Conjugacy separability}

A group $G$ is said to be {\em{conjugacy separable}} if, for
each pair of non-conjugate elements $x,y\in G$ there exists a finite
quotient $p:G\to Q$ such that $p(x)$ is not conjugate to $p(y)$ in $Q$.
Conjugacy separability leads to a solution to the conjugacy problem in a
group in much the same way as residual finiteness leads to a solution
to the word problem. Conjugacy separability is relevant in
the context of the present article because it can be expressed as
 a property of the pair $(G, \hat G)$: it is
equivalent to the statement that $G\hookrightarrow\hat G$ is an embedding and
$x,y\in G$ are conjugate in $\hat G$ if and only if they are conjugate in $G$.

One would like to refine Theorem \ref{t:rips} so as to ensure that the hyperbolic
group $\G$ is conjugacy separable. It is not yet clear if such a refinement
exists\footnote{Update September 2008: 
Owen Cotton-Barratt and Henry Wilton have now proved that it does exist.}. 
If it did, then one could arrange for the group $G$ in  Theorem \ref{t:conj1} to
be conjugacy separable. For the moment,  I can only arrange this by weakening
the finiteness condition on $P$.

\begin{prop} \label{p:weak}
There exist residually finite groups $u:P\hookrightarrow\G$, with $P$ finitely generated and
$\G$ finitely presented,
so that $\hat u:\hat P\to \hat \G$ is an isomorphism but  $\G$ is conjugacy separable
while the conjugacy problem in $P$ is unsolvable.
\end{prop}
 
\begin{proof} Let $Q=\<X\mid R\>$ be as in the proof of Theorem \ref{t:conj1} and let $p:F\to Q$
be the implied surjection from the free group on the set $X$. Let $P\subset F\times F$
be the corresponding fibre product and note that $P$ is generated by the finite set
$\{(x,x)\mid x\in X\}\cup\{(r,1)\mid r\in R\}$.

Arguing exactly as in the proof of Theorem \ref{t:conj1},
we see that $P$ has an unsolvable conjugacy problem and that $P\hookrightarrow F\times F$
induces an isomorphism of profinite completions. Baumslag \cite{baum} proved that free groups
are conjugacy separable, from which it follows easily that $F\times F$ is.
\end{proof}

\begin{remark} Finitely presented subgroups of direct products of free groups have
solvable conjugacy problem \cite{BM}. What is more, using
results from \cite{BHMS} and \cite{BW}, 
Chagas and Zalesskii \cite{CZ} recently proved that all finitely presented residually free
groups are conjugacy separable.  Thus one cannot perturb the proof of Proposition \ref{p:weak}
in a trivial way so as to make $P$  finitely presented.
\end{remark}

\section{On the triviality problem for super-perfect groups}\label{s:triv}

In Section \ref{s:iso} we shall prove Theorem \ref{t:iso} by exploiting
 the existence of sequences of finite presentations with the following properties.

\def\X{\mathcal X}

\begin{thm}\label{t:triv1} There exists a finite 
set $\X$ and a recursive sequence $(\R_n)$ of finite
sets of words in the letters $\X^{\pm 1}$, of   fixed
cardinality,
so that there is no algorithm to
determine which of the groups $Q_n=\<\X\mid  \R_n\>$
are trivial, and each of the groups has the
following properties:
\begin{enumerate}
\item  $H_1(Q_n;\Z) = H_2(Q_n;\Z)=0$;
\item $Q_n$ has a compact classifying space $K(Q_n,1)$;
\item $Q_n$ has no non-trivial finite quotients.
\end{enumerate}
Moreover, if $Q_n\neq 1$ then $Q_n$ has an unsolvable word problem.
\end{thm}

\subsection{The origin of the undecidability}\label{ss:CM}

\begin{thm}[Collins and Miller]\label{aspher}\label{t:CM}
 There is an integer $k$,
a finite set $X$ and a recursive sequence $(R_n)$ of finite
sets of words in the letters $X^{\pm 1}$ so that:
\begin{enumerate}
\item $|R_n|=k$ for all $n$, and $|X|<k$;
\item each of the groups $\E_n=\langle X\mid R_n\>$ is torsion-free;
\item there is no algorithm that can determine which of these
groups are trivial;
\item if $\E_n$ is non- trivial, then the presentation $\Pi_n\equiv\langle X\mid R_n\>$ 
is aspherical and each $x\in X$ is non-trivial in $\E_n$.
\end{enumerate}
\end{thm} 

We sketch the proof. In \cite{CM},
Collins and Miller
explain in detail how Boone's classical example of a finitely
presented group $B$ with an unsolvable word problem can be obtained
from a finitely generated free group in a finite number of steps, where each
step consists of taking a free product with amalgamation or an
HNN extension in which the associated subgroups are free of finite
rank; in particular $B$ is torsion-free and has an aspherical presentation
$\Pi_0$.

Collins and Miller go on to observe that in Miller's interpretation of 
Rabin's construction (\cite{cfm}, p.88), the finite presentations
$\Pi^w$ that are constructed,
indexed by words $w$ in the generators of $\Pi_0$, either
define the trivial group (the case where $w=1$ in $B$) or else
(when $w\neq 1$ in $B$) they are
the natural presentations associated to passing from
$B$ to a group $\E_w$ by a finite chain of free products
with free groups and HNN extensions along finitely generated free
groups; in particular $\E_w$ is torsion-free if $w\neq 1$ in $B$, and
$\Pi^w$ is aspherical. 
In Miller's construction the presentations $\Pi^w$
share a common set of generators (i.e. are defined
as quotients of a fixed
free group) and have the same number of relators. Moreover the number
of relators is greater than the number of generators. If $\E_w$ is non-trivial then
each of  the given generators has infinite order. 

Thus the proof is complete, modulo a switch of 
indexing set to $\mathbb N$, which can be achieved by replacing
each word $w$ by its index in the lex-least ordering, for example.

\subsection{Excluding finite quotients}\label{ss:noFin}
Let the presentations $\Pi_n\equiv\langle X\mid R_n\>$ be as 
 in Theorem \ref{aspher}. We shall describe an
 algorithm that modifies these presentations so as to ensure that the
groups  they present have no proper subgroups of finite index.

Suppose $X=\{x_1,\dots,x_l\}$. 

Let $\langle Y\mid T\rangle$ be a finite aspherical presentation of a group $J$
 that has no finite quotients, and let $y\in Y$
be an element whose normal closure in $J$ is the whole group; that is,
$J/\ln y\rn = 1$. For example, we can take $J$ to be one of the finitely presented
infinite simple groups constructed by Burger and Mozes \cite{BMo}, in which case any $y\neq 1$ will do,
or we can
take the standard presentation of Higman's group 
$\langle a,b,c,d \mid aba^{-1}=b^2,\, bcb^{-1}=c^2,\, 
cdc^{-1}=d^2,\, dad^{-1}=a^2\rangle$ and let $y$  be any of $a,b,c,d$.

\begin{definition}\label{d:Pi'} Let $E_n'$ be the group with presentation
$$
\Pi_n' \equiv \langle X\cup Y_1\cup\dots Y_l
\mid R_n,\, T_1,\dots, T_l,\, x_i^{-1}y_i \ (i=1,\dots,l)\>,
$$
where the $\< Y_i\mid T_i\>$ are disjoint duplicates of
$\< Y \mid T\>$ (our fixed presentation of $J$) with $y_i\in Y_i$ corresponding to $y\in Y$.
\end{definition}

If the group $\E_n$ presented by $\Pi_n$ is non-trivial, then each $x_i$ has infinite order in $\E_n$. In
this case $\Pi_n'$ is the natural  presentation  for the
group obtained from $\E_n$ in $l$ stages
by forming  amalgamated free products
$$\E_{n,1}=\E_{n}\ast_\Z J_1\ \text{  then }\   \E_{n,i}=\E_{n,i-1}\ast_\Z J_i
\ \text{ until } \E_n' := \E_{n,l-1}\ast_\Z J_l,
$$
where the $J_i$ are isomorphic copies of $J$ and the amalgamation
defining $\E_{n,i}$ identifies $x_i\in \E_n\subset \E_{n,i-1}$
with $y\in Y$ in the copy of $J$ that is being
attached.  

On the other hand,
if $\E_n=1$ then $\E_n'=1$, as one sees by making iterated use of the
observation that  the pushout of the diagram
$\{1\}\leftarrow\Z\rightarrow J$ is trivial if the generator of $\Z$ is mapped to $y\in J$,
since $J/\ln y\rn =1$.  

\smallskip
\noindent{\em{Notation:}} Let $\X$ and $\S_n$ be, respectively, the generators
and relators of
$\Pi_n'$ as given in Definition \ref{d:Pi'}.

\begin{lemma}\label{l:ready} Let the groups $\E_n$ and $\E_n'$
and the presentation $\Pi_n'\equiv \<\X\mid \S_n\>$ be as above.
\begin{enumerate}
\item  
$\E_n'$ is trivial if and only if 
$\E_n$ is trivial. 
\item For all $n\in\N$, the group $\E_n'$ has no non-trivial finite quotients.
\item The cardinality of $\S_n$ is independent of $n$, and $|\S_n| > |\X|$.
\item If $\E_n'$ is non-trivial then the presentation $\Pi_n'$ is aspherical.
\end{enumerate}
\end{lemma}

\begin{proof} Item (1)
is covered by the preceding discussion while (3) is manifest in the definition
of $\Pi_n'$. In the light of the observations in Section \ref{ss:ass},
item (4) follows immediately from the above description of $\Pi_n'$ as the presentation of
an iterated free product with cyclic amalgamation, together with the fact that
$\< Y \mid T\>$ and
$\Pi_n$ are aspherical.  

Suppose $\E_n'$ is non-trivial and consider the copies  $J_i$ of $J$ visible
in the repeated amalgamations formed in passing from $\E_n$ to $\E_n'$.
The relations
$x_i=y_i$ show that the union of the $J_i$ generates $\E_n'$. Since $J$ has
no non-trivial finite quotients, neither does $\E_n'$.
\end{proof}

\subsection{Presenting universal central extensions}\label{ss:chuck}

I learnt the  following result from C.F.~Miller III, and I am grateful to
him for letting me reproduce his proof here.

Given a group $\G$, we shall use the standard notation $[A,B]$ to denote the 
subgroup of $\G$ generated by the set of commutators $\{[a,b] : a\in A, b\in B\}$.

\begin{prop}\label{p:miller} Let $G= \< x_1,\dots,x_n \mid r_1,\dots,r_m \>$ be a finitely presented
group, let $F$ be the free group on $\{x_1,\dots,x_n\}$ and let $R\subset F$ be
the normal closure of $\{r_1,\dots,r_m\}$. Suppose that $\G$ is perfect and choose
$c_i\in [F,F]$ such that $x_ic_i\in R$. Then the following is a finite presentation
for the universal central extension of $G$:
$$
\<
x_1,\dots,x_n \mid x_ic_i,\ [x_i,\,r_j] \ (i=1,\dots,n; \ j=1,\dots,m) 
\>.
$$
\end{prop}

\begin{proof}  Let $K\subset F$  be the subgroup that is the normal
closure of
the relators in the above presentation. We must prove that
$F/K$ is isomorphic to $[F,F]/[F,R]$, the universal central extension of $G$.

Note that $[F,R]\subset K$.

Let $\tilde X:=\{x_1c_1,\dots,x_nc_n\}$.
Since $x_ic_i\in R$, the image of $\tilde X$ in $F/[F,R]$ is central.
In particular $K/[F,R]$ is abelian, generated by the image of 
$\tilde X$. 

Since the image of $\tilde X$ generates $F/[F,F]$, the
natural map $K/[F,R]\to F/[F,F]$ is onto. Moreover, 
as the image of $\tilde X$ is a basis for $F/[F,F] \cong \Z^n$, it
must also be a basis for $K/[F,R]$.
Hence the natural map $K/[F,R]\to F/[F,F]$ is an isomorphism. In
particular the kernel of this map is trivial, so
 $K\cap [F,F]\subset [F,R]$. But $[F,R]\subset K$, 
so $K\cap [F,F]\subset [F,R]$.

Now consider the map $[F,F] \to F/K$. As  
$x_ic_i\in K$ and $c_i\in[F,F]$, 
the image of this map contains $x_iK$ for $i=1,\dots,n$. Thus the map
is onto.  Its kernel is $[F,F]\cap K$, which we just proved is
$[F,R]$. Therefore $F/K$ is isomorphic to $[F,F]/[F,R]$.
\end{proof}

\begin{cor}\label{c:algo}
There exists an algorithm that, given a finite presentation $\<X\mid \Sigma\>$
of a perfect group $G$, will output a finite presentation $\< X\mid \tilde\Sigma\>$ for the universal central
extension of $G$.
\end{cor}

\begin{proof} Let $F$ be the free group on $X$.
The algorithm generates an enumeration $d_0,d_1,\dots$ of words representing
the elements of $[F,F]$ and an enumeration $\rho_0,\rho_1,\dots$
of the normal closure of $\Sigma$; it runs through the list of all pairs $(d_i,\rho_j)$,
proceeding along finite diagonals. It works as follows:  
for each $x\in X$ in turn, it runs through the products $xd_i\rho_j$ checking to see if
any of them is equal to the identity in $F$ (that is, freely equal to the empty word). Since
 the given group $G$ is perfect, the algorithm will eventually find indices $i(x)$ and
$j(x)$ such that $xd_{i(x)}\rho_{j(x)}$ is equal to the identity in $F$. The algorithm
outputs  
$$
\tilde\Sigma = \{ xd_{i(x)}: x\in X\} \cup \{ [\sigma,\, x_k] : \sigma\in\Sigma, x\in X\}.
$$
\end{proof}

\subsection{The proof of Theorem \ref{t:triv1}}

In Subsection \ref{ss:noFin} we constructed a recursive sequence of finite
presentations $\Pi_n'\equiv\<\X\mid \S_n\>$ defining quotients $\E_n'$ of a fixed free group $F=F(\X)$, such
that none of the $\E_n'$ have any non-trivial finite quotients but there was no algorithm
to determine which of the $\E_n'$ are trivial;   if $\E_n'\neq 1$ then $\Pi_n'$ is aspherical.
Furthermore,  if $\E_n'\neq 1$ then it has
 an unsolvable word problem, because it contains a copy of the group $B$ with which the
proof of Theorem \ref{aspher} began.

To complete the proof of Theorem \ref{t:triv1} we apply the algorithm of Corollary
\ref{c:algo}  to transform
each of the presentations $\Pi_n'$ into a presentation $\tilde\Pi_n\equiv\<\X\mid \tilde\S_n\>$
for the universal central extension $\tilde \E_n'$ of $\E_n'$.  Note that
the presentations $\tilde\Pi_n$  still
define the groups $\tilde\E_n'$ as quotients of the fixed free group $F(\X)$,
and that $|\tilde\S_n|$ is independent of $n$.

If we define $Q_n:=\tilde \E_n'$ and $\R_n:=\tilde\Sigma_n$, then
Proposition \ref{p:Gtilde} assures us  that the groups $Q_n = \langle \X\mid \R_n\>$ have the properties
stated in  Theorem \ref{t:triv1}.
 \hfill $\square$

\section{The isomorphism problem}\label{s:iso}

The following is a more precise formulation of Theorem \ref{t:iso}.

\begin{thm}\label{t:iso1} There exists a finitely generated free group $F=F(\Y)$
and two recursive sequences of finite subsets $(\T_n)$ and $(\U_n)$, 
with cardinalities independent of $n$, such that:
\begin{enumerate}
\item for all $n$, the group $G_n:=\<\Y\mid \T_n\>$ is finitely presented, residually
finite, and has a compact classifying space;
\item for all $n$, the subgroup $P_n\subset G_n$ generated by the image of $\U_n$ is finitely presentable;
\item for all $n$, the inclusion $P_n\hookrightarrow G_n$ induces an isomorphism of 
profinite completions $\hat P_n\to \hat G_n$;
\item there is no algorithm that can determine for which $n$ the inclusion $P_n\hookrightarrow G_n$
is an isomorphism;
\item there is no algorithm that can determine for which $n$ the groups $P_n$ and $G_n$ are
abstractly isomorphic.
\end{enumerate}
\end{thm}

\begin{proof}
Let the presentations $\<\X\mid\R_n\>$ 
for the groups $Q_n$ be as in Theorem \ref{t:triv1}. The proof of the present theorem may be
summarized as follows: we take the presentation of
$G_n:=\G_n\times\G_n$ obtained by applying the algorithm of Theorem \ref{t:rips}
to $\<\X\mid\R_n\>$, and we define $\U_n$ to be the natural
generating set for the fibre product $P_n\subset G_n$ of
$p_n:\G_n\to Q_n$. What follows is merely a careful exposition of this construction.

Theorem \ref{t:rips} associates to $\<\X\mid\R_n\>$  a
group $\Gamma_n$ with presentation
$\mathcal G_n\equiv\<\X,a_1,a_2,a_3\mid \cech\R_n\>$.
We define $\Y$ to be the disjoint union of two copies
of $\cech\X=\X\cup\{a_1,a_2,a_3\}$ and denote
words $w$ in the letters of the first (resp.~second) copy $(w,1)$
(resp.~$(1,w)$).

Let $\T_n =  \{(1,r), (r,1) \mid r\in \cech\R_n\}\cup\{[(x,1),\, (1,z)] \mid x,z\in\cech\X\}$ and
note that $\langle \Y \mid \T_n\>$ is a presentation of $G_n:=\G_n\times\G_n$.

Let $\U_n = \{(a_i,1),\, (1,a_i)\mid i=1,2,3\}\cup \{(x,1)(1,x) \mid x\in\cech\X\}$ and note that the image of
$U_n$ in $G_n$
generates  the fibre product $P_n$ of 
the map $p_n:\Gamma_n\to Q_n$, which is defined by $p_n(x,1)=p_n(1,x)=p(x)$ for $x\in\X$ and $p(a_i)=1$
for $i=1,2,3$. 

Because
the derivation of $\T_n$ and $\U_n$ from $\R_n$ is entirely algorithmic,
  the sequences $(\T_n)$ and $(\U_n)$ are recursive.

The Bridson-Grunewald Theorem (\ref{t:BG})  tells us that  
$P_n\hookrightarrow G_n$ induces an isomorphism $\hat P_n\to
\hat G_n$, and   $P_n$ is finitely presentable because $Q_n$ has type $F_3$.
(Recall that $Q_n$ was constructed so as to have a finite classifying space.)

If $Q_n=1$, then $P_n=G_n$. But if $Q_n$ is infinite then the fact that certain centralizers in
$P_n$ are not finitely presented shows that it is not isomorphic to $G_n$ (\cite{BG}, Section 6). 
Theorem \ref{t:triv1} tells us that
there is no algorithm that can determine whether or not $Q_n\neq 1$. Hence there is 
no algorithm that can determine whether or not $P_n$ is isomorphic to $G_n$, and no
algorithm that can determine whether or not $P_n\hookrightarrow G_n$ is
an isomorphism. 
\end{proof}

\section{Super-Perfect Group with Unsolvable Generation Problem}\label{s:gen}

In this Section we prove two theorems, each of which implies that there
exist finitely presented super-perfect groups with no finite quotients
in which there is no algorithm to determine which finite subsets generate.

\subsection{An inability to distinguish super-perfect subgroups from those with free quotients}\label{ss:gen}

\begin{thm}\label{t:gen} There exists a finitely presented 
super-perfect group $\La$,
an integer $r>0$,
and a recursive sequence of $r$-element subsets $\mathcal S_n\subset \La$, 
given by words in the generators, so that each $\mathcal S_n$ {\bf{either}} 
generates $\La$ or {\bf{else}}  generates a subgroup that 
maps onto a non-abelian free group,
 and there is no algorithm that can determine which
alternative holds. Moreover, $\La$ has a compact classifying space
and no proper subgroups of finite index.
\end{thm}

\begin{remark}\label{r:notfp}
An additional feature of the sets $\mathcal S_n$ is that if 
$\<\mathcal S_n\>$
is not equal to $\La$ then it is not finitely presentable, 
but this will play no role in the present article.
\end{remark}

Theorem \ref{t:CM} provides us with an example of a sequence of
finitely presented groups $\E_n = \< X \mid R_n\>$, with $k=|R_n|$
fixed, for which there is no algorithm to determine which of the
$\E_n$ are trivial, but where one knows that each $x\in X$ has infinite
order in $\E_n$ if $\E_n\neq 1$.  

Let $F$ denote the free group on $X$
and recall once again that the fibre product $P_n\subset F\times F$ of the natural projection
$F\to \E_n$ is generated by
$$
S_n =\{(x,x) \mid x\in X\} \cup \{(r,1)\mid r\in R_n\}.
$$
Each $S_n$ has $k+l$ elements, where $l$ is the cardinality of $X=\{x_1,\dots,x_l\}$.
Moreover since the sequence $(R_n)$ is recursive, so is $(S_n)$.

Let $\Delta$ be the amalgamated free product of $F\times F$ with $2l$ copies of Higman's group,
$$
J_i=\langle a_i,b_i,c_i,d_i \mid a_ib_ia_i^{-1}=b_i^2,\, b_ic_ib_i^{-1}=c_i^2,\, 
c_id_ic_i^{-1}=d_i^2,\, d_ia_id_i^{-1}=d_i^2\rangle,
$$ 
where each  $J_i$ is attached to $F\times F$ along a cyclic
subgroup: for $i\le l$ the amalgamation identifies $d_i\in J_i$ with $(x_i,1)$
and for $i>l$ the amalgamation identifies $d_i\in J_i$ with $(1,x_{i-l})$.

Let $\C = \bigcup_{i=1}^{2l}\{a_i,b_i,c_i\}$,  let $S_n^+=S_n\cup\C$
and let $M_n\subset\Delta$ be the subgroup generated by $S_n^+$.

\begin{lemma}\label{l:Mn}
\begin{enumerate}
\item $\Delta$ has no non-trivial finite quotients.
\item  If $\E_n = 1$ then $M_n=\Delta$.
\item If $\E_n\neq 1$ then $M_n$ maps onto a free
group of rank $2l$.
\end{enumerate}
\end{lemma}

\begin{proof} We saw in Proposition \ref{p:hig}  that $J$ has no non-trivial
finite quotients, and the union of the subgroups $J_i\cong J$ generate 
$\Delta$, so (1) is proved.

If $\E_n = 1$ then $S_n$ generates $F\times F$,
and it is clear that $F\times F$ and $\C$ together generate $\Delta$.
This proves (2).

If $\E_n\neq 1$ then the subgroup of $F\times F$ generated by $S_n$
has trivial intersection with each of the cyclic subgroups $\<(x_i,1)\>$
and $\<(1,x_i)\>$. Furthermore, 
in Proposition \ref{p:hig} we proved that the subgroup  
$D_i\subset J_i$ generated by $\{a_i,b_i,c_i\}$ (which is isomorphic to 
$D=\<a,b,c \mid aba^{-1}=b^2,\, bcb^{-1}=c^2\rangle$) 
intersects $\<d_i\>$ trivially. It follows 
(by the standard theory of amalgamated free products) that $\<S_n\>$ and $\<\mathcal C\>$ 
generate their free product
in $\Delta$. Thus $M_n$ is isomorphic to the free product of $P_n$
with $2l$ copies of $D$. We also proved in Proposition \ref{p:hig}
that $D$ maps onto $\Z$, and therefore $M_n$ maps onto a 
free group of rank $2l$.
\end{proof}

\noindent{\bf{Proof of Theorem \ref{t:gen}.}}
Let $\tilde\Delta$ be the universal central extension of $\Delta$.
From Lemma \ref{l:Mn}(1) and Proposition \ref{p:Gtilde}(5), we
know that $\tilde\Delta$ has no proper subgroups of finite index.

We have generators $\mathcal A=
X_\lambda\cup X_\rho\cup \mathcal C$ for $\Delta$,
where $X_\lambda=\{(x,1)\mid x\in X\}$ 
generates the subgroup $F\times\{1\}$
and $X_\rho=\{(1,x)\mid x\in X\}$ generates the subgroup $\{1\}\times F$.
We may regard $S_n$ (and hence $S_n^+$) as a set of words
in these generators by replacing $(x,x)$ in the definition of $S_n$
by $(x,1)(1,x)$ and by identifying $(r,1)$ with the corresponding 
reduced word in the letters $X_\lambda^{\pm 1}$.

We choose a lift $\tilde a\in \tilde\Delta$ of each $a\in\mathcal A$ and
define $\tilde{\mathcal A}=\{\tilde{a} \mid a\in\mathcal A\}$. We then
choose a finite generating set $Z$ for the kernel of $\tilde\Delta
\to\Delta$ and work
with the generating set $\mathcal B =\tilde{\mathcal A}\cup Z$ for
$\tilde\Delta$.

We define $\tilde S_n^+$ to be the set of words in the letters $\mathcal
A^{\pm 1}$ obtained by replacing each $a\in\mathcal A$ by 
$\tilde a$, and we define $\mathcal S_n = \tilde S_n^+\cup Z$.

By construction, the subgroup of  $\tilde\Delta$ generated by 
$\mathcal S_n$ maps onto $M_n\subset\Delta$ and contains
the kernel of $\tilde\Delta\to\Delta$. Thus $\mathcal S_n$ generates
$\tilde\Delta$ if $M_n=\Delta$ (equivalently, $\E_n=1$) and
$\<\mathcal S_n\>$ maps onto a free group of rank $2l$ if
$\E_n\neq 1$. And the sequence $\E_n$ was chosen deliberately
so that there is no algorithm that can determine which of these
alternatives holds. 

Since the standard presentations of $F\times F$ and $J$
are aspherical and $\Delta$ is obtained from these groups by a sequence of amalgamations
along cyclic subgroups, $\Delta$ has an aspherical presentation
(see Section \ref{ss:ass}). It follows from Proposition \ref{p:Gtilde}(8)
that $\tilde\Delta$ also has a compact classifying space.
 \hfill $\square$

\begin{remark} Returning to Remark \ref{r:notfp}, note that since $M_n$
is obtained from $\<\mathcal S_n\>\subset\tilde\Delta$ by factoring
out the finitely generated subgroup $\<Z\>$,
if $\<\mathcal S_n\>$ were finitely presentable then
$M_n$ would be too. But $M_n$
 is not finitely presentable if $\E_n\neq 1$, 
because it contains $P_n\subset F\times F$ as a free
factor, and this is not finitely presentable \cite{grun}. 
\end{remark}

\subsection{An inability to distinguish between subdirect products in an acyclic group}\label{ss:fibProd}

A subgroup $H$ of a direct product $G_1\times G_2$ is termed a
{\em{subdirect product}} if the restriction to $H$ of the coordinate projection
$G_1\times G_2\to G_i$ is onto for $i=1,2$. An important example
of a subdirect product is the fibre product $P\subset G\times G$ of a
surjection $p:G\to Q$. Such fibre products can be characterised as those
subdirect products of $G\times G$ that contain the diagonal.

We need the following variation on Theorem 3.1 of \cite{BM}.

\begin{lemma} \label{l:mbc}
Let $A$ and $B$ be  super-perfect groups, let $P\subset A\times B$
be a subdirect product, and let $L=A\cap P$. Then $H_1(P;\Z)\cong H_2(A/L;\Z)$.
\end{lemma}

\begin{proof} Note that $L:=A\cap P$  is normal in both $A$ and $P$.
The actions of $P$ and $A$ by conjugation on
$L$ define the same subgroup of ${\rm{Aut}}(L)$ and hence the
coinvariants $H_0(P/L; H_1(L,\Z))$ and $H_0(A/L; H_1(L,\Z))$ coincide.
We will prove that the first of these groups is isomorphic to 
$H_1(P,\Z)$ and the second is isomorphic to $H_2(A/L; \Z)$.

The five term exact sequence for $Q=A/L$ gives the exactness of  
$$\cdots H_2(A;\Z) \to H_2(A/L;\Z)\to H_0(A/L;H_1(L;\Z))
 \to H_1(A;\Z) \cdots.$$
Thus, since $A$ is super-perfect, $H_2(A/L; \Z)\cong H_0(A/L;H_1(L;\Z))$.

Similarly, the five term exact sequence for $P/L$ gives the exactness of 
$$\cdots  H_2(P/L;\Z)\to H_0(P/L;H_1(L;\Z))
 \to H_1(P;\Z) \to H_1(P/L;\Z)\to 0.$$
Thus, since $B\cong P/L$ is super-perfect, $H_1(P;\Z)\cong H_0(P/L;H_1(L;\Z))$.
\end{proof}

A discrete group $G$ is said to be {\em{acyclic}} if $H_i(G;\Z)=0$ for all $i\ge 1$. It
follows from the Mayer-Vietoris theorem that a free product of acyclic groups is
acyclic and from the K\"unneth formula that a direct product of acyclic groups is acyclic.
We proved in Proposition \ref{p:hig}(5) that Higman's group $J$ is acyclic.

\begin{thm}\label{t:acyclicFP}
There exists a finitely presented acyclic group $\H$, with no non-trivial
finite quotients, an integer $m$, and
a recursive sequence of $m$-element subsets $\th_n\subset\H\times\H$, given as
words in the generators of $\H\times \H$, with the following properties:
\begin{enumerate}
\item For all $n$, the subgroup $\Th_n$ generated by $\th_n$ is a subdirect
product that contains the diagonal subgroup $\H^{\Delta}:=\{(h,h)\mid h\in\H\}$.
\item For all $n$, the  finite quotients of $\Th_n$ are all abelian.
\item If $\Th_n\neq \H\times \H$ then $H_1(\Th_n;\Z)$  is infinite and torsion free.
\item There is no algorithm that, with input $\th_n$, can determine
whether or not $\Th_n$ is equal
(or isomorphic) to $\H\times \H$.
\end{enumerate}
\end{thm}

\begin{proof} Let $\<Y\mid T\>$ be the standard presentation
for Higman's group $J$ and consider the recursive sequence of presentations
$(\Pi_n')$ for $\E_n'$ given in definition \ref{d:Pi'}. 
 The presence of the relations $x_i^{-1}y_i$
in $\Pi_n'$ shows that the union $\mathbb Y$ of the sets $Y_i$ generates $\E_n'$.

Let $V_n$ be the set of words in the letters $\mathbb Y^{\pm 1}$ obtained
from $R_n$ (a subset of the relations of $\Pi_n'$) by replacing each occurrence of $x_i$
with $y_i$. Note that $(V_n)$ is a recursive sequence 
since $(R_n)$ was. We make the Tietze moves on $(\Pi_n')$ that delete the $x_i$ and the
relations $x_i^{-1}y_i$, thus obtaining the presentation
$\<\mathbb Y\mid V_n, T_1,\dots,T_l\>$ for $\E_n'$.

Let $\H$ be the free product of $l$ copies $J_i = \<Y_i\mid T_i\>$ of $J$
and let $q_n:\H\to\E_n'$ be the epimorphism  implicit in the labelling of generators.
The fibre product $\Th_n:=\{(h_1,h_2)\mid q_n(h_1)=q_n(h_2)\}\subset \H\times \H$ of $q_n$
is generated by $\th_n=\{(y,1)(1,y)\mid y\in\mathbb Y\}\cup \{(v,1)\mid v\in V_n\}$.
When the elements of $\th_n$ are expressed in the obvious manner as words in
the generators  $\{(y,1),\, (1,y) \mid y\in\mathbb Y\}$ of $\H\times\H$, the sets
$(\th_n)$ form a recursive sequence of sets of a fixed cardinality, since $(V_n)$ is
such a sequence.

By construction, $\Th_n = \H\times \H$ if and only if $\E_n'$ is trivial, and the $\E_n'$
were chosen so that there is no algorithm that can determine when $\E_n'=1$.

Since it is a fibre product, $\Th_n\subset \H\times \H$ is subdirect product 
and contains the diagonal subgroup $\H^{\Delta}=\H\times\H$.
Since $\H^{\Delta}\cong \H$ is a free product of copies of Higman's group, it has no non-trivial finite
quotients, so  any finite quotient $G$ of $\Th_n$ is generated by the images of the non-diagonal generators
$(v,1)$. Moreover, these are forced to commute in $G$, because given $(v_1,1)$
and $(v_2,1)$ we have $(v_i,1)(1,v_i)=1$ in $G$ and $(v_i,1)$ commutes with
$(1,v_j^{-1})$ in $\Th_n$. Thus $G$ is abelian.

It only remains to prove that $H_1(\Th_n;\Z)$ is infinite and torsion free if $\E_n'\neq 1$.
But this follows immediately from Lemmas \ref{l:mbc} and  \ref{l:H2} because
$\E_n'$ has an aspherical presentation with more relations than generators, namely $\Pi_n'$.
\end{proof}

\section{The proof of Theorem \ref{t:inj}}

\def\bY{\mathbb Y}
\def\bT{\mathbb T}
 
We maintain the notation of the preceding proof. 
We fix a finite presenation $\H=\<\bY \mid \bT\>$, where $\bY$ is the disjoint union
of the $Y_i$, as above. 
By applying the
algorithm of Theorem \ref{t:rips} to this presentation we obtain a short
exact sequence $1\to N\to \G\overset{p}\to \H\to 1$ and a presentation
for $\G$ with generators $\bY\cup\{a_1,a_2,a_3\}$, where $\{a_1,a_2,a_3\}$
generates $N$. 

As generators for $\G\times\G$ we take two disjoint
copies of this generating set, writing $(u,1)$ for words in the first and $(1,u)$
for words in the second.

Theorem \ref{t:BG} assures us that the fibre product $P\subset \H\times\H$
is finitely presented and that $\hat P\to \hat\H\times\hat\H$ is an isomorphism.
 
Let $\pi:\G\times\G\to \H\times\H$ denote the map induced by $p$.
By definition, $P=\pi^{-1}(\H^\Delta)$ where $\H^\Delta$ is
the diagonal subgroup. Let $K_n = \pi^{-1}(\Theta_n)$. Note that
$P\subset K_n$, by Theorem \ref{t:acyclicFP}(1), and that $K_n$
is generated by $\tilde\theta_n:=\theta_n\cup\A$,
where  $\A=\{(a_i,1),\, (1,a_i)\mid i=1,2,3\}$. Since $(\theta_n)$
is a recursive sequence, so is $(\tilde\theta_n)$.
 
If $\Theta_n=\H\times\H_n$ then $K_n=\G\times\G$. If 
$\Theta_n\neq 1$ then  Theorem \ref{t:acyclicFP}(3) tells us that
that there is a map, $q$ say, from $\Theta_n$ onto an infinite abelian group $Z$.
The composition $q\circ\pi$ maps $K_n$ onto $Z$ but 
$q\circ\pi(P)=1$ because $\pi(P)\cong\H$ is perfect.
Thus if $\Theta_n\neq 1$ we obtain infinitely many finite quotients of 
$K_n$ in which the image of $P$ is trivial; in particular $\hat P\to \hat K_n$
is not surjective. Since $\hat P\to \hat\G\times\hat\G$ is an isomorphism, it
follows that $\hat K_n\to \hat\G\times\hat\G$ is not injective.

We proved in Theorem \ref{t:acyclicFP} that there is no algorithm that can
determine when $\Theta_n\neq 1$, so the proof of Theorem \ref{t:inj} is complete.
\hfill $\square$

\end{document}